\documentclass{amsart}
\usepackage{epsfig}
\usepackage{url}	

\newtheorem{theorem}{Theorem}[section]
\newtheorem{lemma}[theorem]{Lemma}

\theoremstyle{definition}

\theoremstyle{remark}

\numberwithin{equation}{section}

\begin{document}

\title[Multiple harmonic star sums]{Multiple harmonic star sums involving
central binomial coefficients as weights}

\author{Karl Dilcher}
\address{Department of Mathematics and Statistics\\
 Dalhousie University\\
         Halifax, Nova Scotia, B3H 4R2, Canada}
\email{dilcher@mathstat.dal.ca}

\author{Christophe Vignat}
\address{CentraleSup\'elec, Universit\'e Paris-Saclay, Gif-sur-Yvette, France and Department of
Mathematics, Tulane University, New Orleans, LA 70118, USA}
\email{cvignat@tulane.edu}
\thanks{The first author was supported in part by the Natural Sciences and 
Engineering Research Council of Canada.}


\setcounter{equation}{0}

\begin{abstract}
We consider weighted analogues of multiple harmonic and multiple $t$-harmonic
star sums, where the ``star" indicates non-strict inequalities between the 
summation indices. In the special case where all exponents are 2, which is of
particular interest also for ordinary multiple ($t$-)harmonic star sums, we 
obtain evaluations in terms of simple finite sums. The proofs are based on
several combinatorial identities which are of interest in their own right.
\end{abstract}

\maketitle

\section{Introduction}\label{sec:1}

The ordinary harmonic numbers, which are usually denoted by 
\begin{equation}\label{1.1}
H_n:=\sum_{k=1}^n\frac{1}{k}\qquad (n\geq 1),
\end{equation}
belong to the most important objects in number theory, combinatorics and
other parts of pure and applied mathematics. Among various 
generalizations are the multiple harmonic sums, which themselves occur in 
different variants. The ordinary {\it multiple harmonic sum} is 
defined by
\begin{equation}\label{1.2}
H_n(s_1,\ldots,s_r):=\sum_{1\leq k_1<\cdots<k_r\leq n}
\frac{1}{k_1^{s_1}\cdots k_r^{s_r}},
\end{equation}
and its ``star version" by
\begin{equation}\label{1.3}
H_n^*(s_1,\ldots,s_r):=\sum_{1\leq k_1\leq\cdots\leq k_r\leq n}
\frac{1}{k_1^{s_1}\cdots k_r^{s_r}},
\end{equation}
The odd analogues of the sums have recently been called {\it multiple
$t$-harmonic sums},
\begin{equation}\label{1.4}
t_n(s_1,\ldots,s_r):=\sum_{1\leq k_1<\cdots<k_r\leq n}
\frac{1}{(2k_1-1)^{s_1}\cdots(2k_r-1)^{s_r}},
\end{equation}
along with their star version $t_n^*(s_1,\ldots,s_r)$. To all these concepts
there are alternating versions and $q$-analogues which have been studied in the
recent literature. Finally, the limits as $n\to\infty$ lead to related
versions of multiple zeta values, which have also been of great interest in
recent years. Further details on all this can be found, for instance, in 
\cite{HHT} or \cite{LY1}. 

In the various studies of the four main types of multiple harmonic sums
mentioned above, the exponent $s_j=2$ plays a particularly important role; 
see again \cite{HHT} and \cite{LY1} 
and relevant extensions and generalizations such as \cite{Zh} and \cite{LY2}.
In the present paper we will mainly consider the case where {\it all} exponents
are $s_1=\cdots =s_r=2$. Throughout this paper we use the standard notation 
$\{2\}^r$ for the $r$-tuple $(2,\ldots,2)$. 

We first remark that the sums $H_n(\{2\}^r)$ and $t_n(\{2\}^r)$ occur in the
Maclaurin series expansions of positive integer powers of inverse 
trigonometric function; see Section~\ref{sec:4} below. 

We now turn to the star versions of the multiple harmonic and multiple
$t$-harmonic sums. Only in recent years it was shown that there are explicit
evaluations in several infinite classes. They are, however, still quite 
restricted and mainly involve the exponents $s_j=2$ in the denominators; see 
again \cite{HHT}, \cite{LY1}, \cite{LY2}, and \cite{Zh}. When all exponents 
are 2, these evaluations become especially simple and pleasing. In particular, 
as a consequence of Theorem~2.3 in \cite{HHT} we have
\begin{equation}\label{1.5}
H_n^*(\{2\}^r)=2\sum_{k=1}^n\frac{(-1)^{k-1}\binom{n}{k}}{\binom{n+k}{k}k^{2r}},
\end{equation}
and from the special case $d=0$ in \cite[p.~1102]{LY1} it follows that
\begin{equation}\label{1.6}
t_n^*(\{2\}^r)=\frac{n\binom{2n}{n}}{2^{4n-3}}
\sum_{k=1}^n\binom{2n-1}{nk}\frac{(-1)^{k-1}}{(2k-1)^{2r}}.
\end{equation}

It is the main purpose of this paper to prove the following variants
of the identities \eqref{1.5} and \eqref{1.6}.

\begin{theorem}\label{thm:1.1}
For all positive integers $n$ and $r$ we have
\begin{equation}\label{1.7}
\sum_{1\leq k_1\leq\cdots\leq k_r\leq n}\frac{4^{k_1}}{\binom{2k_1}{k_1}}\cdot
\frac{1}{k_1^2\cdots k_r^2} = \frac{4}{\binom{2n}{n}}
\sum_{k=0}^{\lfloor\frac{n-1}{2}\rfloor}\frac{\binom{2n}{n-2k-1}}{(2k+1)^{2r}}.
\end{equation}
\end{theorem}

The left-hand side of \eqref{1.7} can be seen as a weighted analogue of the
multiple harmonic star sum in \eqref{1.3}. Similarly, the next result involves
a weighted analogue of the multiple $t$-harmonic star sum.

\begin{theorem}\label{thm:1.2}
For all positive integers $n$ and $r$ we have
\begin{equation}\label{1.8}
\sum_{1\leq k_1\leq\cdots\leq k_r\leq n}\frac{\binom{2k_1}{k_1}k_1}{4^{k_1}} 
\cdot\frac{1}{(2k_1-1)^2\cdots(2k_r-1)^2} = \frac{\binom{2n}{n}n}{4^{2n-1}}
\sum_{k=1}^n\frac{\binom{2n-1}{n-k}}{(2k-1)^{2r}}.
\end{equation}
\end{theorem}

To the best of our knowledge, weighted analogues of multiple harmonic or
$t$-harmonic star sums have rarely been considered before. We mention two such
instances in Section~\ref{sec:4} below. Somewhat related is the
``mollified companion" of $H_n(s_1,\ldots,s_r)$, defined by identity~(6) in
\cite[p.~1671]{Zh}. However, the weight $\binom{n}{k_r}/\binom{n+k_r}{k_r}$
(here in our notation) occurs in the largest summation index, as opposed to 
the smallest one in \eqref{1.7} and \eqref{1.8}, and no evaluations are 
involved.

For the proofs of Theorems~\ref{thm:1.1} and~\ref{thm:1.2} we require a few
combinatorial identities, which will be stated and proved in 
Section~\ref{sec:2}. We will then prove the theorems in Section~\ref{sec:3} and
conclude the paper with some further remarks in Section~\ref{sec:4}.

\section{Some combinatorial identities}\label{sec:2}

The proofs of the theorems will be done by induction on $n$, and for both the
induction beginnings and the induction steps we require the following
combinatorial identities.

\begin{lemma}\label{lem:2.1}
For any integer $n\geq 1$ we have
\begin{equation}\label{2.1}
\sum_{k=1}^n\frac{4^k}{\binom{2k}{k}k^2} = \frac{4}{\binom{2n}{n}}
\sum_{k=0}^{\lfloor\frac{n-1}{2}\rfloor}\frac{\binom{2n}{n-2k-1}}{(2k+1)^2}.
\end{equation}
\end{lemma}

\begin{proof}
We equate two different evaluations of the same class of infinite series.
From Corollary~5 in \cite{DV1}, we have
\begin{equation}\label{2.2}
\sum_{k=1}^\infty\frac{4^k}{\binom{2k}{k}k(2k+2n)}=\frac{\binom{2n}{n}}{4^{n+1}}
\left(\pi^2+2\sum_{k=1}^n\frac{4^k}{\binom{2k}{k}k^2}\right).
\end{equation}
On the other hand, the identity (3.6) in \cite{DV2} with $p=1$ evaluates
the left-hand side of \eqref{2.2} as
\begin{equation}\label{2.3}
\frac{\binom{2n}{n}}{2^{2n+1}}\cdot\frac{\pi^2}{2}
+\frac{1}{4^n}\sum_{k=1}^n\frac{\binom{2n}{n-k}}{k^2}
-\frac{1}{4^n}\sum_{k=1}^n(-1)^k\frac{\binom{2n}{n-k}}{k^2}.
\end{equation}
The last two terms in \eqref{2.3} combine to give
\[
\frac{2}{4^n}\sum_{k=0}^{\lfloor\frac{n-1}{2}\rfloor}
\binom{2n}{n-2k-1}\frac{1}{(2k+1)^2}.
\]
Finally, by equating the rational parts in \eqref{2.2} and \eqref{2.3} we get
\eqref{2.1}, as desired.
\end{proof}

\begin{lemma}\label{lem:2.2}
For any integers $1\leq \ell\leq n$ we have
\begin{equation}\label{2.4}
\sum_{k=\ell}^n\frac{\binom{2k}{k-\ell}}{\binom{2k}{k}k^2} 
= \frac{\binom{2n}{n-\ell}}{\binom{2n}{n}\ell^2}.
\end{equation}
\end{lemma}

\begin{proof}
This can be proven by induction on $n$ with base case $n=\ell$. However, an easy
manipulation of binomial coefficients shows that \eqref{2.4} is equivalent to
the identity (2.4) in \cite{HHT}.
\end{proof}

\begin{lemma}\label{lem:2.3}
For any integer $n\geq 1$ we have
\begin{equation}\label{2.5}
\sum_{k=1}^n\frac{\binom{2k}{k}k}{4^k(2k-1)^2}=\frac{\binom{2n}{n}n}{4^{2n-1}}
\sum_{k=1}^n\frac{\binom{2n-1}{n-k}}{(2k-1)^2}.
\end{equation}
\end{lemma}

\begin{proof}
Using Corollary~5 in \cite{DV1} with $n$ replaced by $2n+1$, we get
\begin{equation}\label{2.6}
\sum_{k=1}^\infty\frac{4^k}{\binom{2k}{k}k(2k+2n+1)}
=\frac{2^{2n+1}}{\binom{2n}{n}(2n+1)}
\sum_{k=0}^n\frac{\binom{2k}{k}}{4^k(2k+1)}.
\end{equation}
On the other hand, using (3.5) in \cite{DV2} with $p=1$ and $\ell$ replaced by
$n$, we have
\begin{equation}\label{2.7}
\sum_{k=1}^\infty\frac{4^k}{\binom{2k}{k}k(2k+2n+1)}
=\frac{2}{4^n}\sum_{k=0}^n\frac{\binom{2n+1}{n-k}}{(2k+1)^2}.
\end{equation}
We now equate \eqref{2.6} and \eqref{2.7} and shift the summations by 1,
obtaining
\begin{equation}\label{2.8}
\sum_{k=1}^{n+1}\frac{\binom{2k-2}{k-1}}{4^{k-1}(2k-1)}
=\frac{\binom{2n}{n}(2n+1)}{4^{2n}}
\sum_{k=1}^{n+1}\frac{\binom{2n+1}{n+1-k}}{(2k-1)^2}.
\end{equation}
Finally, we replace $n$ by $n-1$ in \eqref{2.8} and use the identity
\begin{equation}\label{2.9}
\binom{2k-2}{k-1} = \binom{2k}{k}\frac{k}{2(2k-1)}
\end{equation}
both in this form and with $n$ in place of $k$. Then \eqref{2.8} leads to 
\eqref{2.5} after some easy manipulations.
\end{proof}

\begin{lemma}\label{lem:2.4}
For any integers $1\leq \ell\leq n$ we have
\begin{equation}\label{2.10}
\sum_{k=\ell}^n\frac{\binom{2k}{k}\binom{2k-1}{k-\ell}k}{4^{2k}(2k-1)^2}
= \frac{\binom{2n}{n}\binom{2n-1}{n-\ell}n}{4^{2n}(2\ell-1)^2}.
\end{equation}
\end{lemma}

\begin{proof}
We proceed by induction on $n$, with $\ell\geq 1$ fixed. The base case $\ell=1$
is clear since both side are then the same. For what follows, we require again
the binomial identity \eqref{2.9}, as well as
\begin{equation}\label{2.11}
\binom{2n-1}{n-\ell} 
= \binom{2n+1}{n+1-\ell}\frac{(n-\ell+1)(n-\ell)}{2n(2n+1)},
\end{equation}
which is easy to verify with some basic properties of the binomial 
coefficients.

We now assume that \eqref{2.10} holds for some $n\geq\ell$. For the induction
step we add to both side of \eqref{2.10} the term in the sum on the left that
corresponds to $k=n+1$, obtaining
\begin{equation}\label{2.12}
\sum_{k=\ell}^{n+1}\frac{\binom{2k}{k}\binom{2k-1}{k-\ell}k}{4^{2k}(2k-1)^2}
= \frac{\binom{2n}{n}\binom{2n-1}{n-\ell}n}{4^{2n}(2\ell-1)^2}
+\frac{\binom{2n+2}{n+1}\binom{2n+1}{n+1-\ell}(n+1)}{4^{2n+2}(2n+1)^2}.
\end{equation}
We now use \eqref{2.9} with $k=n+1$ and \eqref{2.11} to rewrite the right-hand
side of \eqref{2.12} as
\begin{align*}
\frac{\binom{2n+2}{n+1}\binom{2n+1}{n+1-\ell}(n+1)}{4^{2n+2}(2\ell-1)^2(2n+1)^2}
&\left(4(n-\ell+1)(n+\ell)+(2\ell-1)^2\right) \\
&\quad=\frac{\binom{2n+2}{n+1}\binom{2n+1}{n+1-\ell}(n+1)}{4^{2n+2}(2\ell-1)^2},
\end{align*}
which is the right-hand side of \eqref{2.10} with $n$ replaced by $n+1$. This
completes the induction.
\end{proof}

\section{Proofs of the theorems}\label{sec:3}

\begin{proof}[Proof of Theorem~\ref{thm:1.1}]
We proceed by induction on $r$. The base case $r=1$ is simply 
Lemma~\ref{lem:2.1}. Let now $r\geq 2$ and assume that Theorem~\ref{thm:1.1}
holds for $r-1$. We denote the left-hand side of \eqref{1.7} by $L_1$ and for
greater ease of notation we set $j=k_r$. By induction hypothesis we then have
\begin{align}
L_1&=\sum_{j=1}^n\frac{1}{j^2}\sum_{1\leq k_1\leq\cdots\leq k_{r-1}\leq j}
\frac{4^{k_1}}{\binom{2k_1}{k_1}}\cdot\frac{1}{k_1^2\cdots k_{r-1}^2}\label{3.1}\\
&=\sum_{j=1}^n\frac{1}{j^2}\cdot\frac{4}{\binom{2j}{j}}
\sum_{k=0}^{\lfloor\frac{j-1}{2}\rfloor}\frac{\binom{2j}{j-2k-1}}{(2k+1)^{2r-2}}\nonumber\\
&=\sum_{k=0}^{\lfloor\frac{n-1}{2}\rfloor}\frac{4}{(2k+1)^{2r-2}}
\sum_{j=2k+1}^n\frac{\binom{2j}{j-2k-1}}{\binom{2j}{j}j^2}.\nonumber
\end{align}
We now apply Lemma~\ref{lem:2.2} to the right-most sum, with $\ell=2k+1$ and
$k=j$. Then with \eqref{3.1} we get
\[
L_1=\sum_{k=0}^{\lfloor\frac{n-1}{2}\rfloor}\frac{4}{(2k+1)^{2r-2}}
=\frac{\binom{2n}{n-2k-1}}{\binom{2n}{n}(2k+1)^2}.
\]
This is the same as the right-hand side of \eqref{1.7}, which concludes the 
induction step and thus the proof of Theorem~\ref{thm:1.1}.
\end{proof}

\begin{proof}[Proof of Theorem~\ref{thm:1.2}]
We proceed again by induction on $r$. Here the base case $r=1$ is clearly
Lemma~\ref{lem:2.3}. As before, we let $r\geq 2$ and assume that 
Theorem~\ref{thm:1.2} holds for $r-1$. Denoting the left-hand side of 
\eqref{1.8} by $L_2$ and setting $j=k_r$, we get by induction hypothesis 
\begin{align}
L_2&=\sum_{j=1}^n\frac{1}{(2j-1)^2}\sum_{1\leq k_1\leq\cdots\leq k_{r-1}\leq j}
\frac{\binom{2k_1}{k_1}k_1}{4^{k_1}}
\cdot\frac{1}{(2k_1-1)^2\cdots(2k_{r-1}-1)^2}\label{3.2}\\
&=\sum_{j=1}^n\frac{1}{(2j-1)^2}\cdot\frac{\binom{2j}{j}j}{4^{2j-1}}
\sum_{k=1}^j\frac{\binom{2j-1}{j-k}}{(2k-1)^{2r-2}}\nonumber\\
&=\sum_{k=1}^n\frac{1}{(2k-1)^{2r-2}}
\sum_{j=k}^n\frac{\binom{2j}{j}\binom{2j-1}{j-k}j}{4^{2j-1}(2j-1)^2}.\nonumber
\end{align}
To deal with the right-most sum, we use Lemma~2.4, getting with \eqref{3.2},
\[
L_2=\sum_{k=1}^n\frac{1}{(2k-1)^{2r-2}}
\cdot\frac{\binom{2n}{n}\binom{2n-1}{n-k}n}{4^{2n-1}(2k-1)^2},
\]
which is the right-hand side of \eqref{1.8}. This concludes the induction step
and the proof of Theorem~\ref{thm:1.2}.
\end{proof}

\section{Further remarks}\label{sec:4}

{\bf 1.} The following class of weighted multiple harmonic star sums was 
obtained in Lemma~5.5 of \cite{TZ}, which we quote with a notation consistent
with \eqref{1.7} above: For all positive integers $n$ and $r$ and variable $x$,
we have
\begin{equation}\label{5.1}
\sum_{1\leq k_1\leq\cdots\leq k_r\leq n}\frac{(1+x)^{k_1}-1}{k_1\cdots k_r}
= \sum_{k=1}^n\binom{n}{k}\frac{x^k}{k^r}.
\end{equation}
The case $x=-1$ was previously known and can be stated in analogy to 
\eqref{1.5} as
\begin{equation}\label{5.2}
H_n^*(\{1\}^r)=\sum_{k=1}^n\binom{n}{k}\frac{(-1)^k}{k^r}.
\end{equation}
As was the case with \eqref{1.5}, the identity \eqref{5.2} was also generalized
in \cite{HHT}. A second identity, related to \eqref{5.1}, can be found as
Lemma~5.6 in \cite{TZ}:
\begin{equation}\label{5.3}
\sum_{1\leq k_1\leq\cdots\leq k_r\leq n}
\frac{(1+x)^{k_1}}{k_1\cdots k_r}(-1)^{k_r}\binom{n}{k_r}
=\sum_{k=1}^n\frac{x^k-1}{k^r}.
\end{equation}

{\bf 2.} As mentioned in the Introduction, the sums $H_n(\{2\}^r)$ and 
$t_n(\{2\}^r)$ occur in the Maclaurin series expansions of positive integer 
powers of inverse trigonometric functions;
see the classical book by Schwatt \cite[p.~123ff]{Sch},
which was first published in 1924. For instance (see also \cite{BC}), with the
notations in \eqref{1.2} and \eqref{1.4}, we have for all integers $r\geq 0$,
\begin{align*}
(\arcsin{x})^{2r+1} &= (2r+1)!\sum_{k=0}^\infty
\frac{\binom{2k}{k}t_k(\{2\}^r)}{4^k}\cdot\frac{x^{2k+1}}{2k+1},\\
(\arcsin{x})^{2r+2} &= \frac{(2r+2)!}{4^r}\sum_{k=1}^\infty
\frac{4^k H_{k-1}(\{2\}^r)}{\binom{2k}{k}2k}\cdot\frac{x^{2k}}{2k},
\end{align*}
where $t_k(\{2\}^0)=H_k(\{2\}^0)=1$ for all $k\geq 0$.

\medskip
{\bf 3.} The identities \eqref{2.4} and \eqref{2.10} are quite special in that
the summands on the left are very similar to the sums on the right. This
indicates that there may be a general principle behind this phenomenon. And
indeed, an analysis of the situation leads to the following result.

\begin{theorem}\label{thm:4.1}
For each integer $\ell\geq 1$, let $(a_{j,\ell})_{j\geq\ell}$ be a sequence
of real or complex numbers such that $a_{j,\ell}\neq 1$ for $j\geq\ell+1$. Then
for all integers $1\leq\ell\leq n$, we have the identity
\begin{equation}\label{4.4}
\sum_{k=\ell}^n\prod_{j=\ell}^{k-1}\frac{a_{j,\ell}}{a_{j+1,\ell}-1}
=a_{n,\ell}\prod_{j=\ell}^{n-1}\frac{a_{j,\ell}}{a_{j+1,\ell}-1},
\end{equation}
where $a_{\ell,\ell} = 1$ for all $\ell\geq 1$.
\end{theorem}

If we set $a_{j,\ell}=j^2/\ell^2$, resp.\ $(2j+1)^2/(2\ell+1)^2$, then after
some straightforward manipulations we recover the identities \eqref{2.4} and 
\eqref{2.10}, respectively.

It would be outside the scope of this paper to provide a proof of 
Theorem~\ref{thm:4.1} and give further results and applications. This can be
found in the forthcoming publication \cite{DV3}.

\end{document}